\documentclass[12pt]{amsart}
\usepackage{amsthm,amsmath,amsfonts,amssymb,latexsym,amscd,mathrsfs,hyperref,cleveref,graphicx}
\usepackage{setspace}
\usepackage{cite}
\usepackage{phonetic}
\setdisplayskipstretch{2}
\usepackage{cite}
\usepackage{crossreftools}

\makeatletter
\renewcommand{\pod}[1]{\allowbreak\mathchoice
  {\if@display \mkern 18mu\else \mkern 8mu\fi (#1)}
  {\if@display \mkern 18mu\else \mkern 8mu\fi (#1)}
  {\mkern4mu(#1)}
  {\mkern4mu(#1)}
}

\makeatletter
\@namedef{subjclassname@2020}{%
  \textup{2020} Mathematics Subject Classification}
\makeatother

\theoremstyle{plain}
\newtheorem{thm}{Theorem}
\newtheorem{lem}{Lemma}

\newtheorem*{thm*}{Theorem}
\newtheorem*{cor*}{Corollary}
\newtheorem*{prop*}{Proposition}

\theoremstyle{definition}

\newtheorem*{exa*}{Example}

\Crefname{thm}{Theorem}{theorems}
\Crefname{lem}{Lemma}{lemmas}

\newcommand{\mb}{\mathbb}
\newcommand{\mc}{\mathcal}
\newcommand{\mf}{\mathfrak}
\renewcommand{\mod}{\operatorname{mod}}

\def \a{\alpha}  \def \d{\delta} \def \e{\varepsilon}  \def \k{\kappa}  \def \s{\sigma} \def \t{\theta} 

\makeatletter
\def\widebreve{\mathpalette\wide@breve}
\def\wide@breve#1#2{\sbox\z@{$#1#2$}%
     \mathop{\vbox{\m@th\ialign{##\crcr
\kern0.08em\brevefill#1{0.8\wd\z@}\crcr\noalign{\nointerlineskip}%
                    $\hss#1#2\hss$\crcr}}}\limits}
\def\brevefill#1#2{$\m@th\sbox\tw@{$#1($}%
  \hss\resizebox{#2}{\wd\tw@}{\rotatebox[origin=c]{90}{\upshape(}}\hss$}
\makeatletter

\numberwithin{equation}{section}

\renewcommand{\labelenumi}{\setlength{\labelwidth}{\leftmargin}
   \addtolength{\labelwidth}{-\labelsep}
   \hbox to \labelwidth{\theenumi.\hfill}}

\begin{document}
\title{Smooth numbers in Beatty sequences}
\author{Roger Baker}
\address{Department of Mathematics\newline
\indent Brigham Young University\newline
\indent Provo, UT 84602, U.S.A}
\email{baker@math.byu.edu}

 \begin{abstract}
Let $\t$ be an irrational number of finite type and let $\psi \ge 0$. We consider numbers in the Beatty sequence of integer parts,
 \[\mc B(x) = \{\lfloor\t n + \psi\rfloor : 
 1 \le n \le x\}.\]
Let $C > 3$. Writing $P(n)$ for the largest prime factor of $n$ and $|\ldots|$ for cardinality, we show that
 \[|\{n\in \mathcal B(x) : P(n) \le y\}| =
 \frac 1\t\, \Psi(\t x, y)\ (1 + o(1))\]
as $x\to\infty$, uniformly for $y \ge (\log x)^C$. Here $\Psi(X,y)$ denotes the number of integers up to $X$ with $P(n) \le y$. The range of $y$ extends that given by Akbal \cite{akb}. The work of Harper \cite{harp} plays a key role in the proof.
 \end{abstract}



\maketitle

 \section{Introduction}\label{sec:intro}
 
A positive integer $n$ is said to be \textit{$y$-smooth} if $P(n)$, the largest prime factor of $n$, is at most $y$. We write $\mf S(y)$ for the set of $y$-smooth numbers in $\mb N$ and
 \[\Psi(x,y) = |\mf S(y) \cap [1,x]|,\]
where $|\ldots|$ denotes cardinality.

Let $\t > 1$ be an irrational number and $\psi \in [0,\infty)$. Arithmetic properties of the Beatty sequence
 \[\mc B(x) = \{\lfloor\t n + \psi\rfloor :
 1 \le n \le x\}\]
(where $\lfloor$\ \ $\rfloor$ denotes integer part), have been studied in \cite{akb, bakzhao, bankshpar, har}, for example. One may conjecture that for $x$ large and $y=y(x)$ not too small, say $x \ge y \ge (\log x)^C$ where $C > 1$, we have
 \begin{equation}\label{eq1.1}
|\{n\in \mc B(x) : P(n) \le y\}| = \frac 1\t\, \Psi (\t x, y) \quad (1 + o(1))
 \end{equation}
where $o(1)$ denotes a quantity tending to 0 as $x$ tends to infinity. Banks and Shparlinski \cite{bankshpar} obtained \eqref{eq1.1} (in a slightly different form) uniformly for
 \[\exp((\log x)^{2/3+\e}) \le y \le x.\]
(We write $\e$ for an arbitrary positive number.) Under the additional condition that $\t$ is of finite type, Akbal \cite{akb} obtained \eqref{eq1.1} uniformly for
 \begin{equation}\label{eq1.2}
\exp((\log\log x)^{5/3+\e}) \le y \le x.
 \end{equation}
In the present paper we extend the range \eqref{eq1.2}.

 \begin{thm}\label{thm:irrationalnofinitetype}
Let $\t > 1$ be an irrational number of finite type and $\psi \ge 0$. Then \eqref{eq1.1} holds uniformly for
 \[(\log x)^{3+\e} \le y \le x.\]
 \end{thm}

We recall that an irrational number $\t$ is said to be of finite type if ($\|\ldots\|$ denoting distance to the nearest integer) we have
 \[\|m\t\| \ge \frac c{m^\k} \quad (m \in \mb N)\]
for some $c > 0$ and $\k > 0$. We note that if $\t$ is of finite type, then so is $\t^{-1}$.

Theorem \ref{thm:irrationalnofinitetype} depends on an estimate for the exponential sum
 \[S(\t)= \sum_{\substack{n \le x\\[1mm]
 n\in\, \mf S(y)}} e(n\t).\]
Akbal \cite{akb} uses the estimate of Foury and Tenenbaum \cite{fouten}: for $3 \le y \le \sqrt x$, $q \in \mb N$, $(a,q)=1
$, and $\d \in\mb R$,
 \[S\left(\frac aq + \d\right) \ll x (1 + |\d x|)
 \log^3x \left(\frac{y^{1/2}}{x^{1/4}} + \frac 1{q^{1/2}}
 + \left(\frac{qy}x\right)^{1/2}\right).\]
This is unhelpful when, say, $y = (\log x)^C$ $(C > 1)$ since the trivial bound for $S(\t)$ is $x^{1-1/C +o(1)}$ (see e.g. \cite{hilten}).

We use a procedure of Harper \cite[Theorem 1]{harp} and to state his result we introduce some notation. For $2 \le y \le x$, let $u = \frac{\log y}{\log x}$ and let $\a = \a(x,y)$ be the solution of 
 \[\sum_{p \le y} \, \frac{\log p}{p^\a - 1} = \log x.\]
For convenience, when $\t = \frac aq + \d$ as above, we write
 \[L = 2(1 + |\d x|)\]
and
 \begin{equation}\label{eq1.3}
M = u^{3/2} \log u \log x (\log L)^{1/2}(\log q y)^{1/2}.
 \end{equation}
In Theorem 1 of \cite{harp} it is shown that whenever
 \begin{equation}\label{eq1.4}
q^2 L^2 y^3 \le x, 
 \end{equation}
we have
 \begin{equation}\label{eq1.5}
S\left(\frac aq + \d\right) \ll \Psi(x,y)(q(1+|\d x|)^{-\frac 12 + \frac 32 (1 - \a(x,y))}M.
 \end{equation}
We cannot use this bound directly since \eqref{eq1.3} is too restrictive. We adapt Harper's argument to obtain

 \begin{thm}\label{thm:multiplicfunc}
Let $f$ be a completely multiplicative function, $|f(n)|\le 1$ $(n \in \mb N)$. Let
 \[S(f,\t) = \sum_{\substack{n \le x\\[1mm]
 n\in \mf S(y)}} f(n) e(\t n).\]
 \end{thm}
Let $q \in \mb N$, $(a,q) = 1$, and $\d \in \mb R$. Then, with $\a = \a(x,y)$, we have
 \[S\left(f,\frac aq + \d\right) \ll \Psi(x,y) \left\{
 (q(1+|\d x|)^{-\frac 12 + \frac 32\, (1 - \a)}
 M + x^{\a/2}(q Ly^3)^{\frac 12} \sqrt{\log y\log q}\right\}.\]

To save space, we refer frequently to \cite{harp} in our proof of Theorem \ref{thm:multiplicfunc} in Section 2. Theorem \ref{thm:irrationalnofinitetype} is deduced in a straightforward manner from Theorem \ref{thm:multiplicfunc} in Section 3.

The factor $f(n)$ in Theorem \ref{thm:multiplicfunc} is not needed elsewhere in the paper, but it requires no significant effort to include it.

I would like to thank Adam Harper for helpful comments concerning his proof of \eqref{eq1.5}.

\section{Proof of Theorem \ref{thm:multiplicfunc}.}\label{sec:proofthm2}

 \begin{lem}\label{lem:2leyl3x}
Let $2 \le y \le x$ and $d \ge 1$. Then we have 
 \[\Psi \left(\frac xd, y\right) \ll \frac 1{d^{\a(x,y)}}\
 \Psi(x,y).\]
 \end{lem}
 
 \begin{proof}
See de la Bret\`eche and Tenenbaum \cite[Th\'eor\`eme 2.4 (i)]{bretten}
 \end{proof}

We write $p(n)$ for the smallest prime factor of $n\in \mb N$. We begin the proof of Theorem \ref{thm:multiplicfunc} by noting that the result is trivial for $qLy^3 \ge x^\a$. Suppose now that $qLy^3 < x^\a$. Every $y$-smooth number in $[qLy^2,x]$ can be written uniquely in the form $mn$, where
 \[qLy < m \le qLy^2, \ \frac m{P(m)} \le qLy, \ 
 m \in \mf S(y)\]
and
 \[\frac{qLy^2}m \le n \le \frac xm \ , \
 p(n) \ge P(m), \ \ n \in \mf S(y).\]
(We take $m$ to consist of the product of the smallest prime factors of the number.) With $\t = \frac aq + \d$, we have
 \begin{align}
S(f,\t) &= \sum_{\substack{qLy^2 \le n \le x\\[1mm]
n \in \mf S(y)}} f(n) e(n\t) + O(\Psi(qLy^2, y))\label{eq2.1}\\
&= U + O(\Psi(qLy^2,y))\notag
 \end{align}
where
 \[U = \sum_{\substack{
 qLy < m \le qLy^2\\[1mm]
 m/P(m) \le qLy\\[1mm]
 m\in \mf S(y)}} \ \sum_{\substack{
 \frac{qLy^2}m \le n \le \frac xm\\[1mm]
 p(n) \ge P(m)\\[1mm]
 n \in \mf S(y)}} f(mn) e(mn\t).\]
We now decompose $U$ as
 \begin{equation}\label{eq2.2}
U = \sum_{0 \le j \le \frac{\log y}{\log 2}} U_j
 \end{equation}
where
 \[U_j = \sum_{p \le y} \ \ \sum_{\substack{
 2^j qLy < m \le q L y \min(2^{j+1},p)\\[1mm]
 P(m) = p}} f(m) \ \sum_{\substack{
 \frac{qLy^2}m \le n \le \frac xm\\[1mm]
 p(n) \ge p\\[1mm]
 n \in \mf S(y)}} f(n) e (mn\t),\]
noting that if $P(m) = p \le y$, then $m$ is $y$-smooth, and the condition $\frac m{P(m)} \le qLy$ can be written as $m \le qLyp$.

We apply the Cauchy-Schwarz inequality to $U_j$. Let $\sum\limits_m$ denote
 \[\sum_{\substack{
 2^j qLy < m \le qLy\, \min(2^{j+1},p).\\[1mm]
 P(m) = p}}\]
We obtain
 \begin{align*}
U_j &\le \sqrt{\sum_{p\le y} \ \sum_m 1} \
 \sqrt{
 \sum_{2^j \le p \le y}\ \,
 \sum_{\frac{2^jq Ly}p < m' \le \frac{qLy}p\,
 \min(2^{j+1},p)}\Bigg|\sum_{\frac{qLy^2}{m'p} \le
 n \le \frac x{m'p}, p(n) \ge p, n \in \mf S(y)}
 f(n) e(m'pn\t)\Bigg|^2}\\[2mm]
 &\ll \sqrt{\Psi(2^{j+1}qLy, y)} \ \, \sqrt{
 \sum_{2^j \le p \le y}\ \, \sum_{\substack{
 n_1, n_2 \le \frac x{2^jqLy}\\[1mm]
 p(n_1), p(n_2) \ge p\\[1mm]
 n_1, n_2 \in \mf S(y)}} \ \min \left\{\frac{2^{j+1}qLy}p, 
 \frac 1{\|(n_1-n_2)p\t\|}\right\}}
 \end{align*}
For the last step, we open the square and sum the geometric progression over $m'$. We may restrict the sum over primes to $p \ge 2^j$, since otherwise the sum over $m$ is empty. Our final bound here for $U_j$ exactly matches \cite{harp}.

Let
 \[\mf T_j(r) : = \max_{1\le b\le r-1} \
 \sum_{\substack{
 n_1, n_2 \le \frac x{2^jqLy}\\[1mm]
 n_1, n_2 \in \mf S(y)\\[1mm]
 n_1 - n_2 \equiv b \mod r}} \, 1.\]
Just as in \cite{harp}, after distinguishing the cases $p\mid q$ and $p\nmid q$, we arrive at
 \begin{equation}\label{eq2.3}
U_j \ll \sqrt{\Psi(2^{j+1}q Ly, y)} \, (\sqrt S_1 + \sqrt S_2), 
 \end{equation}
with
 \begin{align*}
S_1 = S_1(j): = \sum_{\substack{
2^j \le p\le y\\[1mm]
p\nmid q}} \ & \mf T_j(q) \sum_{b=1}^{q-1} \min \left\{
\frac{2^{j+1}qLy}p\, , \frac qp\right\}\\[2mm]
&+ \sum_{\substack{2^j \le p \le y\\[1mm]
p\mid q}} \mf T_j \left(\frac qp\right) \ \sum_{b=1}^{(q/p)-1}
\min \left\{\frac{2^{j+1}qLy}p , \frac q{pb}\right\}
 \end{align*}
and
 \begin{gather*}
S_2 = S_2(j) : = \sum_{\substack{
2^j \le p \le y\\[1mm]
p\nmid q}} \frac 1p \ \sum_{\substack{
n_1,n_2\le \frac x{2^jqLy}\\[1mm]
n_1-n_2 \equiv 0 \mod q\\[1mm]
n_1,n_2 \in \mf S(y)}} \min\left\{2^{j+1}qLy, \frac 1{|(n_1-n_2)\d|}\right\}\\[2mm]
\sum_{\substack{
2^j \le p \le y\\[1mm]
p\mid q}} \frac 1p \ \sum_{\substack{
n_1, n_2 \le \frac x{2^jqLy}\\[1mm]
n_1-n_2 \equiv 0\mod{q/p}\\[1mm]
n_1, n_2\in \mf S(y)}} \min\left\{2^{j+1}qLy, \frac 1{|(n_1-n_2)\d|}\right\}.
 \end{gather*}

We now depart to an extent from the argument in \cite{harp}; we have extra terms in the upper bounds in Lemma \ref{lem:upperbounds} below which arise because we do not have the upper bound \eqref{eq1.4}.

 \begin{lem}\label{lem:upperbounds}
Let $(\log x)^{1.1} \le y \le x^{1/3}$, $q \ge 1$, and $L = 2(1 + |\d x|)$. Then for any $j$, $0 \le j \le \frac{\log y}{\log 2}$, and any prime $p\mid q$, we have
 \begin{align}
\mf T_j(q) &\ll \frac{\Psi(x/2^j qLy, y)^2}q \ q^{1-\a(x,y)} \log x + y \Psi\left(\frac x{2^jqLy}, y\right)\label{eq2.4}\\
\intertext{and}
\mf T_j\left(\frac qp\right) &\ll \frac{\Psi(x/2^j qLy, y)^2}{q/p} \ \left(\frac qp\right)^{1-\a(x,y)}\log x\label{eq2.5}\\
&\hskip 2.25in + y\Psi\left(\frac x{2^jqLy},y\right).\notag
 \end{align}
Under the same hypotheses, we have
 \begin{align}
\sum_{\substack{
n_1,n_2 \le \frac x{2^jqLy}\\[1mm]
n_1-n_2 \equiv 0 \mod q\\[1mm]
n_1, n_2 \in \mf S(y)}} &\min \left\{2^{j+1} qLy, \frac 1{|(n_1-n_2)\d|}\right\}\label{eq2.6}\\[2mm]
&\ll 2^j y\, \Psi \left(\frac x{2^j qLy}, y\right)^2 (qL)^{1-\a(x,y)} \log x \log L \notag\\[2mm]
&\hskip .75in + 2^j qLy^2 \Psi\left(\frac x{2^jqLy},y\right)\notag
 \end{align}
and, for any prime $p\mid q$,

 \begin{align}
\sum_{\substack{
n_1,n_2 \le \frac x{2^jqLy}\\[1mm]
n_1-n_2 \equiv 0 \mod{q/p}\\[1mm]
n_1, n_2 \in \mf S(y)}} &\min \left\{2^{j+1} qLy, \frac 1{|(n_1-n_2)\d|}\right\}\label{eq2.7}\\[2mm]
&\ll p2^j y\, \Psi \left(\frac x{2^j qLy}, y\right)^2 \left(\frac{qL}p\right)^{1-\a(x,y)} \log x \log L \notag\\[2mm]
&\hskip .75in + 2^j qLy^2\, \Psi\left(\frac x{2^jqLy},y\right).\notag
 \end{align}
 \end{lem}

 \begin{proof}
We have
 \[\mf T_j(q) = \max_{1\le b\le q-1} \, 
 \sum_{\substack{
 n_1\le x/2^jqLy\\[1mm]
 n_1\in \mf S(y)}} \ \sum_{\substack{
 n_2 \le x/2^j qLy\\[1mm]
 n_2\equiv n_1-b \mod q\\[1mm]
 n_2 \in \mf S(y)}} 1,\]
and just as in \cite{harp} the inner sum is
 \[\ll y + \frac{\Psi(x/2^jqLy)}q \, 
 q^{1-\a(x,y)} \log x.\]
This leads to the bound \eqref{eq2.4} on summing over $n_1$. The bound \eqref{eq2.5} follows in exactly the same way.

To prove \eqref{eq2.6}, \eqref{eq2.7} we distinguish two cases. If $|\d| \le 1/x$, then $L = 2(1 + |\d x|) \asymp 1$, and the bounds can be proved exactly as above on bounding $\min\left\{2^{j+1} qLy, \frac 1{|(n_1-n_2)\d|}\right\}$ by $2^{j+1} qLy \ll 2^j qy$.

Now suppose $|\d| > 1/x$, so that $L\asymp |\d x|$. We partition the sum in \eqref{eq2.6} dyadically. Let us use $\sum^\dag$ to denote a sum over pairs of integers $n_1$, $n_2 \le \frac x{2^jqLy}$ that are $y$-smooth and satisfy $n_1- n_2 \equiv 0\mod q$. Then we have 
 \begin{align*}
\sideset{}{^\dag}\sum_{|n_1 - n_2| \le \frac x{2^j qL^2y}} &\min\left\{2^{j+1} qLy, \frac 1{|(n_1-n_2)\d|}\right\}\\
&\ll 2^j qLy \sum_{\substack{
n_1 \le \frac x{2^j qLy}\\[1mm]
n_1 \in \mf S(y)}} \ \sum_{\substack{
|n_2 - n_1| \le \frac x{2^j qL^2y}\\[1mm]
n_2 \in \mf S(y)\\[1mm]
n_2\equiv n_1 \mod q}} 1.
 \end{align*}
Following \cite{harp}, but as above \textit{not} absorbing a term $y$, the last expression is
 \begin{align*}
&\ll 2^j qLy \sum_{\substack{
n_1 \le \frac x{2^jqLy}\\[1mm]
n_1\in \mf S(y)}} \left\{\frac{\Psi(x/2^jqLy,y)}{qL}\, (qL)^{1-\a} \log x + y\right\}\\[2mm]
&\ll 2^j y \Psi(x/2^j qLy, y)^2 (qL)^{1-\a} \log x + 2^j q Ly^2 \Psi(x/2^j qLy, y).
 \end{align*}
Similarly, for any $r$, $0 \le r \le (\log L)/\log 2$, we have
 \begin{align*}
&\sideset{}{^\dag}\sum_{\frac{2^r}{2^jqL^2y} < |n_1-n_2|\le \frac{2^{r+1}x}{2^jqL^2y}} \min\left\{2^{j+1} qLy, \frac 1{|(n_1-n_2)\d|}\right\}\\[2mm]
&\hskip 1in \ll \frac{2^jqLy}{2^r} \ \sum_{\substack{
n_1 \le \frac x{2^jqLy}\\[1mm]
n_1\in \mf S(y)}} \ \sum_{\substack{
|n_2 - n_1| \le \frac{2^{r+1}x}{2^jqL^2y}\\[1mm]
n_2\in \mf S(y)\\[1mm]
n_2\equiv n_1\mod q}} 1\\[2mm]
&\qquad\ll \frac{2^jqLy}{2^r} \ \sum_{\substack{
n_1 \le \frac x{2^jqLy}\\[1mm]
n_1 \in \mf S(y)}} \left\{\Psi\left(\frac x{2^jqLy},y\right)\left(\frac{2^r}{qL}\right)^\a \log x + y\right\}\\[2mm]
&\qquad\ll 2^j y \Psi\left(\frac x{2^jqLy}, y\right)^2 \left(\frac{2^r}{qL}\right)^\a \log x + \frac{2^jqLy^2}{2^r} \, \Psi\left(\frac x{2^jqLy}, y\right).
 \end{align*}
The bound \eqref{eq2.6} follows on summing over $r$.

The bound \eqref{eq2.7} follows in exactly the same way; we lose a factor $p^\a$ in the first term on the right-hand side in \eqref{eq2.7} because of the weaker congruence condition. This completes the proof of Lemma \ref{lem:upperbounds}.
 \end{proof}

We now assemble our bounds to prove Theorem \ref{thm:irrationalnofinitetype}. As noted in \cite{harp}, for any $p \le y$ we have
 \begin{equation}\label{eq2.8}
\sum_{b=1}^{q-1} \min\left\{\frac{2^{j+1}qLy}p, \frac qb\right\} \ll q \log q 
 \end{equation}
and, for $p\mid q$,
 \begin{equation}\label{eq2.9}
\sum_{b=1}^{(q/p)-1} \min\left\{\frac{2^{j+1}qLy}p, \frac q{pb}\right\} \ll \frac qp \log q.
 \end{equation}

Let
\[A_j = \Bigg(\sum_{\substack{
2^j \le p \le y\\[1mm]
p\nmid q}} \frac 1p + \sum_{\substack{
2^j \le p \le y\\[1mm]
p\mid q}} 1\Bigg) 2^jy \Psi \left(
\frac x{2^jqLy}\, , y\right) (qL)^{1-\a} \log x \log L.\]
We deduce from \eqref{eq2.1}--\eqref{eq2.3}, Lemma \ref{lem:upperbounds}, and \eqref{eq2.8}--\eqref{eq2.9} that
 \begin{equation}\label{eq2.10}
S\left(f, \frac aq +\d\right) \ll \Psi(qLy^2, y) + A + B_1 + B_2 
 \end{equation}
where
 \begin{align*}
A &= \sum_{0 \le j \le \frac{\log y}{\log 2}} \sqrt{\Psi(2^{j+1}qLy,y)} \sqrt{\frac y{\log y}\, \log q \Psi\left(\frac x{2^jqLy},y\right)^2 q^{1-\a} \log x}\\[2mm]
&\quad + \sum_{0 \le j \le \frac{\log y}{\log 2}} \sqrt{\Psi(2^{j+1}qLy, y)} \sqrt{A_j},\\[2mm]
B_1 &= \sum_{0 \le j \le \frac{\log y}{\log 2}} \sqrt{\Psi(2^{j+1}q L y, y)} \ \sqrt{\sum_{2^j \le p \le y} q \log q \cdot y \Psi\left(\frac x{2^jqLy}, y\right)},\\
\intertext{and}
B_2 &= \sum_{0 \le j \le \frac{\log y}{\log 2}} \sqrt{\Psi(2^{j+1}qLy, y)} \sqrt{\frac{qLy^3}{\log y}\, \Psi\left(\frac x{2^jqy}, y\right)}.
 \end{align*}
For the estimation of $A$ we can appeal to \cite{harp}:
 \begin{equation}\label{eq2.11}
A \ll \Psi(x,y)(qL)^{-\frac 12 + \frac 32\, (1-\a(x,y))}M. 
 \end{equation}
The $\Psi$ functions in $B_1$ and $B_2$ are estimated using Lemma \ref{lem:2leyl3x}. Thus
 \begin{align*}
\sqrt{\Psi\left(\frac x{2^jqLy}, y\right)\Psi(2^{j+1}qLy,y)} &\ll \Psi(x,y)(2^jqLy)^{-\a/2} \left(\frac x{2^{j+1}qLy}\right)^{-\a/2}\\[2mm]
&\ll \Psi(x,y)x^{-\a/2},
 \end{align*}
leading to (a slightly wasteful) bound
 \begin{equation}\label{eq2.12}
B_1 + B_2 \ll \Psi(x,y) \sqrt{y^3 qL\log q \log y}\ x^{-\a/2}.
 \end{equation}
The remaining term to be estimated is $\Psi(qLy^2, y)$, and Lemma \ref{lem:2leyl3x} gives
 \begin{equation}\label{eq2.13}
\Psi(qLy^2,y) \ll \left(\frac x{qLy^2}\right)^{-\a} \Psi(x,y).
 \end{equation}
This term can be absorbed into the right-hand side of \eqref{eq2.12} since
 \begin{equation}\label{eq2.14}
(qLy^3)^{\a - \frac 12} \ll (qLy^3)^{\a/2} \ll x^{\a/2}.
 \end{equation}
Theorem \ref{thm:multiplicfunc} follows on combining \eqref{eq2.10}--\eqref{eq2.14}.
 
 \section{Proof of Theorem \ref{thm:multiplicfunc}}\label{sec:proofthm1}
 
 \begin{lem}\label{lem:letu1dotsu}
Let $u_1,\ldots, u_N \in \mb R$. Then for any $J \in \mb N$ and any $\rho \le \s \le \rho + 1$, we have
 \begin{align*}
\big| |\{1 \le n \le N &: u_n \in [\rho,\s] \mod 1\}| - (\s-\rho)N\big|\\
&\qquad \le \frac N{J+1} + 3 \sum_{j=1}^J \ \frac 1j \, \Bigg|\sum_{n=1}^N e(ju_n)\Bigg|.
 \end{align*}
 \end{lem}

 \begin{proof}
See e.g. \cite{rcb}, Theorem 2.1.
 \end{proof}
 
 \begin{proof}[Proof of Theorem \ref{thm:irrationalnofinitetype}]
In view of the result of \cite{akb} cited above, we may suppose that
 \begin{equation}\label{eq3.1}
(\log x)^{3+\e} \le y \le \exp((\log \log x)^2). 
 \end{equation}
We note that $\lfloor \t n + \psi\rfloor = m$ for some $m$ if and only if
 \begin{equation}\label{eq3.2}
0 < \left\{\frac{m+1-\psi}\t\right\} \le \frac 1\t, 
 \end{equation}
so that, applying Lemma \ref{lem:letu1dotsu},
 \begin{align*}
\sum_{n\in \mc B(x,y)} 1 &= \sum_{\substack{
m \le \t x\\[1mm]
m\in \mf S(y)\\[1mm]
\eqref{eq3.2}\, holds}} 1 + O(1)\\[2mm]
&= \t^{-1} \Psi(\t x, y) + O\Bigg(
\frac{\Psi(\t x, y)}{\log x} + \sum_{j=1}^{[\log x]} \ \frac 1j \Bigg|\sum_{\substack{
m \le \t x\\[1mm]
m\in \mf S(y)}} e\left(\frac{jm}\t\right)\Bigg|\Bigg) 
 \end{align*}
For our purposes, then, it suffices to show that
 \begin{equation}\label{eq3.3}
\sum_{\substack{
m \le \t x\\[1mm]
m\in \mf S(y)}} e(j\t^{-1}m) \ll \Psi(\t x,y)(\log x)^{-2}
 \end{equation}
uniformly for $1 \le j \le \log x$. We deduce this from Theorem \ref{thm:multiplicfunc}.

By Dirichlet's theorem there is $q \in \mb N$, $1 \le q \le x^{1/2}$, and $a \in \mb N$ with $(a,q)= 1$ such that
 \[\d : = j\t^{-1} - \frac aq \in \left[-
 \frac 1{qx^{1/2}}, \frac 1{qx^{1/2}}\right].\]
Now
 \[|qj\t^{-1}-a| \ge \frac{c_1}{(qj)^\k}\]
for positive constants $c_1$ and $\k$, hence
 \[x^{-1/2} \ge q|\d| \ge \frac{c_1}{(qj)^\k}\]
giving
 \[q \gg x^{1/2\k}(\log x)^{-1}.\]
We apply Theorem \ref{thm:multiplicfunc} with $j\t^{-1}$ in place of $\t$, and $\t x$ in place of $x$. Note that
 \[qL = 2q(1 + |\d \t x|) \ll x^{1/2}\]
so that, in view of \eqref{eq3.1},
 \begin{equation}\label{eq3.4}
qLy^3 \ll x^{5/9}.
 \end{equation}
Now, since $y \ge (\log x)^{3 + \e}$, we have
 \[\a(\t x, y) \ge \frac 23 + \eta \quad
 (\eta = \eta(\e) > 0);\]
see \cite{hilten}, for example. Thus, abbreviating $\a(\t x, y)$ to $\a$,
 \[q^{-\frac 12 + \frac 32\, (1-\a)} \ll q^{-\eta}
 \ll x^{-\eta/3\k}.\]
Now, with $M$ as in \eqref{eq1.3} with $\t x$ in place of $x$,
 \begin{equation}\label{eq3.5}
q^{-\frac 12 + \frac 32\, (1-\a)}M \ll x^{-\eta/4\k}.
 \end{equation}
Next, recalling \eqref{eq3.4},
 \begin{equation}\label{eq3.6}
(\t x)^{-\a/2}(qLy^3)^{1/2}\sqrt{\log y\log q} \ll x^{-1/20}.
 \end{equation}
Combining \eqref{eq3.5}, \eqref{eq3.6} we obtain \eqref{eq3.3}. This completes the proof of Theorem \ref{thm:irrationalnofinitetype}.
 \end{proof}

 \end{document}